\documentclass[reqno,12pt]{amsart}
\usepackage{geometry}
\usepackage{amsmath,amssymb,mathrsfs,color,stackengine}
\usepackage[colorlinks,
linkcolor=red,
anchorcolor=green,
citecolor=blue,
]{hyperref}
\usepackage[upint]{stix}
\usepackage{subfigure}
\usepackage{float}
\usepackage{times}
\usepackage{tikz}
\usetikzlibrary{intersections}
\usepackage{paralist}

\usepackage{comment}

\makeatletter
\def\setliststart#1{\setcounter{\@listctr}{#1}%
  \addtocounter{\@listctr}{-1}}
\makeatother

\makeatletter
\@addtoreset{figure}{section}
\makeatother

\usepackage{calc}
 \newtheorem{The}{Theorem}[section]
 \newtheorem{Cor}[The]{Corollary}
 \newtheorem{Lem}[The]{Lemma}
 \newtheorem{Pro}[The]{Proposition}
 \theoremstyle{definition}
 \newtheorem{defn}[The]{Definition}
 \newtheorem{Rem}[The]{Remark}
 
 \numberwithin{equation}{section}

\newcommand{\R}{\mathbb{R}}

\newcommand{\bxi}{\boldsymbol{\xi}}
\newcommand{\bu}{\boldsymbol{u}}
\newcommand{\bL}{\boldsymbol{L}}
\newcommand{\bx}{\boldsymbol{x}}
\newcommand{\bv}{\boldsymbol{v}}
\newcommand{\ba}{\boldsymbol{a}}

\newcommand{\bA}{\boldsymbol{A}}

\newcommand{\bp}{\boldsymbol{p}}

\newcommand{\bphi}{\boldsymbol{\phi}}
\newcommand{\cdotss}{\cdots\cdots\cdots\cdots\cdots\cdots\cdots\cdots\cdots\cdots\cdots\cdots\cdots\cdots\cdots\cdots}

\title[Multi-objective Herglotz variational principle]{Multi-objective Herglotz' variational principle and cooperative Hamilton-Jacobi systems}
\author{Wei Cheng, Kai Zhao and Min Zhou}
\address{Department of Mathematics, Nanjing University, Nanjing 210093, China}
\email{chengwei@nju.edu.cn}
\address{School of Mathematical Sciences, Fudan University, Shanghai 200433, China}
\email{zhao\_kai@fudan.edu.en}
\address{Department of Mathematics, Nanjing University, Nanjing 210093, China}
\email{minzhou@nju.edu.cn}
\date{\today}
\subjclass[2010]{35F21, 49L25, 37J50}
\keywords{Herglotz variational principle, weakly coupled systems, Hamilton-Jacobi equation, viscosity solution}
\begin{document}
\maketitle

\begin{abstract}
We study a multi-objective variational problem of Herglotz' type with cooperative linear coupling. We established the associated Euler-Lagrange equations and the characteristic system for cooperative weakly coupled systems of Hamilton-Jacobi equations. We also established the relation of the value functions of this variational problem with the viscosity solutions of cooperative weakly coupled systems of Hamilton-Jacobi equations. Comparing to the previous work in stochastic frame, this approach affords a pure deterministic explanation of this problem under more general conditions. We also showed this approach is valid for general linearly coupling matrix for short time.
\end{abstract}

\section{Introduction}

This paper is devoted to the study of the weakly coupled systems of Hamilton-Jacobi equations using a variational approach based on a vectorial form of the Herglotz' principle. For technical reason, we will mainly discuss this problem under a general setting but a model with wide interests. We suppose our Lagrangians has the form
\begin{equation}\label{eq:L_intro}
	L^i(x,v)-\sum^n_{j=1}a^i_ju_j,\qquad i=1,\ldots,n,
\end{equation}
where $\bA=(a^i_j)$ is an $n\times n$ real matrix and each $L^i:\R^n\times\R^n\to\R$ is a Tonelli Lagrangian, i.e., each $L^i$ is of class $C^2$ and each $L^i(x,\cdot)$ is strictly convex and uniformly superlinear. We impose the following conditions on the coupling matrix $\bA$:
\begin{enumerate}[(C1)]
	\item $a^i_j\leqslant0$ for all $i\not=j$. 
	\item The matrix $\bA$ is irreducible.
\end{enumerate}
We remark that Condition (C1) means the matrix $-\bA$ satisfies so called \emph{Kamke-M\"uller condition} or \emph{cooperativeness condition} in the literature (see \cite{Kamke1932,Muller1927, Smith_book,Coppel_book,Walter_book1970}). We give the definition of irreducible matrix in Appendix. Please refer to Remark \ref{rem:conditions}  and Section 5 for the discussion on more general conditions about $\bA$.

For any $x,y\in\R^n$ and $t>0$, we set
\begin{align*}
	\Gamma^t_{x,y}=\{\xi\in W^{1,1}([0,t],\R^n): \xi(0)=x,\xi(t)=y\}.
\end{align*}
From the theory of ordinary differential equation in the sense of Carath\'eodory (\cite{Coddington_Levinson_book}), the following equation
\begin{equation}\label{eq:Cara_intro}
	\left\{
	\begin{matrix}
		\dot{u}_1(\xi,s)=L^1(\xi(s),\dot{\xi}(s))-\sum^n_{j=1}a^1_ju_j(\xi,s),\\
		\dot{u}_2(\xi,s)=L^2(\xi(s),\dot{\xi}(s))-\sum^n_{j=1}a^2_ju_j(\xi,s),\\
		\cdots\cdots\cdots\cdots\cdots\cdots\cdots\cdots\cdots\cdots \\
		\dot{u}_n(\xi,s)=L^n(\xi(s),\dot{\xi}(s))-\sum^n_{j=1}a^n_ju_j(\xi,s),
	\end{matrix}
	\right.\qquad\qquad s\in[0,t],
\end{equation}
with initial conditions $u_i(0)=a_i$, $i=1,\ldots,n$, admits a unique solution $(u_1(\xi,,\cdot),\ldots,u_n(\xi,,\cdot))$ for each $\xi\in\Gamma^t_{x,y}$. Similar to the problem of the scalar case considered in \cite{CCWY2019,CCJWY2020}, we impose the variational problem
\begin{align*}
	\min\big\{u_i(\xi,t): \xi\in \Gamma^t_{x,y}\big\}.
\end{align*}

It is convenient to write the problem in a compact form. Let
\begin{align*}
	\bu=\begin{pmatrix}
		u_1\\
		u_2\\
		\vdots\\
		u_n
	\end{pmatrix},\quad
	\bx=\begin{pmatrix}
		x_1\\
		x_2\\
		\vdots\\
		x_n
	\end{pmatrix},\quad
	\bv=\begin{pmatrix}
		v_1\\
		v_2\\
		\vdots\\
		v_n
	\end{pmatrix},\quad
	\ba=\begin{pmatrix}
		a_1\\
		a_2\\
		\vdots\\
		a_n
	\end{pmatrix},\quad
	\bxi=\begin{pmatrix}
		\xi\\
		\xi\\
		\vdots\\
		\xi
	\end{pmatrix}
\end{align*}
and
\begin{align*}
	\bL(\bx,\bv)=\begin{pmatrix}
		L^1(x_1,v_1)\\
		L^2(x_2,v_2)\\
		\vdots\\
		L^n(x_n,v_n)
	\end{pmatrix}.
\end{align*}
Therefore \eqref{eq:Cara_intro} becomes
\begin{align*}
	\dot{\bu}=\bL(\bxi,\dot{\bxi})-\bA\cdot\bu,\quad \bu(0)=\ba.
\end{align*}
The uniqueness of the solution allows us to solve the equation above into
\begin{equation}\label{eq:vector_form}
	\bu(t)=e^{-\bA t}\cdot\ba+\int^t_0e^{\bA(s-t)}\cdot\bL(\bxi,\dot{\bxi})\ ds,\quad \bu(0)=\ba,
\end{equation}
or
\begin{align*}
	u_i(\xi,t)=\sum^n_{j=1}b^i_j(t)a_j+\int^t_0\sum^n_{k,j=1}b^i_k(t)c^k_j(s)L^j(\xi,\dot{\xi})\ ds
\end{align*}
where $e^{-\bA s}=(b^i_j(s))$ and $e^{\bA s}=(c^i_j(s))$. Our conditions (C1) and (C2) ensure that each integrand is a time-dependent Tonelli Lagrangian. Thus, we can deduce the existence and regularity of the minimizers for the functional $u^i(\xi,t)$ by classical Tonelli's theorem. Notice under our conditions there is no Lavrentiev phenomenon.

Now, we formulate main results of this paper. We always suppose conditions (C1) and (C2) except (d).
\begin{enumerate}[(a)]
	\item For each $i$, the functional $u_i(\xi,t)$ admits a solution and each minimizer $\xi_i$ is of class $C^2$. Moreover, the system of minimizers $(\xi_1,\ldots,\xi_n)$ satisfies the following Euler-Lagrange equations (Herglotz' equations)
	\begin{align*}
		\frac d{ds}L^i_v(\xi_i(s),\dot{\xi}_i(s))=L^i_x(\xi_i(s),\dot{\xi}_i(s))-\sum^n_{j=1}a^i_jL^j_v(\xi_i(s),\dot{\xi}_i(s)),\qquad i=1,\ldots,n.
	\end{align*}
	\item Let $\phi_1,\ldots,\phi_n\in BUC(\R^n,\R)$, the space of bounded uniformly continuous functions, then the value functions of the Bolza problem of Herglotz' type, defined by
	\begin{align*}
		u^i(t,x)=\inf_{\xi\in\mathcal{A}_{t,x}}\bigg\{\phi_i(\xi(0))+\int^t_0\big(L^i(\xi(s),\dot{\xi}(s))+\sum^n_{j=1}a^i_ju_j(\xi,s)\big)\ ds\bigg\},\quad i=1,\ldots,n,
	\end{align*}
	satisfies the weakly coupled system of Hamilton-Jacobi equations
	\begin{align*}
		\begin{cases}
		    D_tu^1(t,x)+H^1(x,D_xu^1(t,x))+\sum^n_{j=1}a^1_ju^j(t,x)=0,\\
		    \cdotss\\
		    D_tu^n(t,x)+H^n(x,D_xu^n(t,x))+\sum^n_{j=1}a^n_ju^j(t,x)=0,
	    \end{cases}
	\end{align*}
	with initial conditions $u^i(0,x)=\phi_i(x)$, in the sense of viscosity. 
	\item Each value function above also has the representation
	\begin{align*}
		u^i(t,x)=\inf_{\xi\in\mathcal{A}_{t,x}}\bigg\{\sum^n_{j=1}b^i_j(t)\phi_j(\xi(0))+\int^t_0\sum^n_{k,j=1}b^i_k(t)c^k_j(s)L^j(\xi(s),\dot{\xi}(s))\ ds\bigg\}.
	\end{align*}
	\item If $\bA$ is an arbitrary real matrix with condition (C3) in Section 5, then all the statements in (a), (b) and (c) hold true for small time (see Theorem \ref{thm:short_time}).
\end{enumerate}

This paper is motivated by two types of work from different approaches. From PDE point of view, there have been a lot of works on this topic. The existence and uniqueness results for weakly coupled systems of Hamilton-Jacobi equations have been established by \cite{Engler_Lenhart1991,Ishii_Koike1991}. We also note many works such as \cite{CLLN2012,Mitake_Tran2012} on long time behavior, \cite{MSTY2016,Davini_Siconolfi_Zavidovique2018} on the random nature, \cite{Ishii_Jin2020,Ishii2021} on the vanishing discount problem and \cite{Davini_Zavidovique2014} on the Aubry set. We especially emphasize on the paper \cite{Mitake_Tran2014} by Matake and Tran. In their paper, from the previous work in the stochastic frame of this problem in \cite{Davini_Zavidovique2014}, they deduced, for a concrete model, a new representation formula for the viscosity solutions of the weakly coupled system of Hamilton-Jacobi equations. This representation formula has quite interesting information on the deterministic nature of the problem. A recent work by Jin, Wang and Yan \cite{Jin_Wang_Yan2020} also affords some new observation from the recent development of scalar Hamilton-Jacobi equations of contact type.

In the recent work on the weak KAM theory of contact type, there are mainly two ways to study the associated scalar Hamilton-Jacobi equation of contact type. One is the implicit variational principle developed in \cite{Wang_Wang_Yan2017,Wang_Wang_Yan2019_1,Wang_Wang_Yan2019_2} and the other is the generalized variational principle of Gustav Herglotz (\cite{Guenther_Guenther_Gottsch1995,Giaquinta_Hildebrandt_book_II, Bravetti2017,CCWY2019,CCJWY2020,HCHZ2021}). In \cite{Jin_Wang_Yan2020}, the authors obtained a representation formula for the solution of the Cauchy problem with very general non-linear coupling, using the idea from implicit variational principle in the scalar case. 
But, one cannot find dynamical systems behind this problem there. 

However, the dynamics behind may help us to understand the problem in a much deeper way, which is the main theme of this paper. We developed the characteristic system for the weakly coupled systems and gave a very clear connection between the Hamilton-Jacobi theory and the Lagrange-Hamilton theory in a complete deterministic way. We emphasize this approach indeed tells us some new observation from the viewpoint of dynamical system, geometry, mathematical physics and dynamic game theory under general conditions. The general theory of multi-contact system will be our future work. 

The paper is organized as follows. In section 2, we raise a multi-objective variational problem of Herglotz' type and obtain the regularity and some necessary conditions of optimality such as Euler-Lagrange equations. We also established the characteristic system of this problem. In Section 3, we consider the relevant Bolza problem and prove the value functions of the associated problem is the unique viscosity solutions of the associated weakly coupled systems of Hamilton-Jacobi equations. In Section 4, We introduce the Lax-Oleinik evolution of this cooperative Hamilton-Jacobi systems and discuss the relation between positive and negative type Lax-Oleinik evolutions. In Section 5, we obtained similar results on small time interval, removing condition (C1) and (C2) but adding condition (C3) on the Lagrangians. There is a short appendix on some basic knowledge on the cooperativeness condition of the coupling matrix.

\medskip

\noindent\textbf{Acknowledgements.} Wei Cheng is partly supported by National Natural Science Foundation of China (Grant No. 11871267, 11631006 and 11790272). Min Zhou is partly supported by National Natural Science Foundation of China (Grant No. 11790272 and 11631006). The authors thank Qinbo Chen, Jiahui Hong and Tianqi Shi for helpful discussion. The first author also thanks Hitoshi Ishii for the suggestion of this problem during his visit to Tokyo in 2019.

\section{A Herglotz-type variational problem for systems}

For each $i=1,\ldots,n$, let $L^i:\R^n\times\R^n\times\R^n\to\R$ be of class $C^2$ satisfying the following conditions:
\begin{enumerate}[(L1)]
	\item $L^i(x,v,u_1,\ldots,u_n)$ is a strictly convex in the $v$-variable for all $x\in\R^n$ and $(u_1,\ldots,u_n)\in\R^n$.
	\item There exist two superlinear nondecreasing functions $\overline{\theta},\theta:[0,+\infty)\to[0,+\infty)$, $\theta(0)=0$ and $c_1>0$, such that
	\begin{align*}
		\overline{\theta}_i(|v|)\geqslant L^i(x,v,0,\ldots,0)\geqslant\theta_0(|v|)-c_1,\quad (x,v)\in\R^n\times\R^n.
	\end{align*}
	\item There exists $K>0$ such that
	\begin{align*}
		|L^i_{u_j}(x,v,u_1,\ldots,u_n)|\leqslant K,\quad (x,v,u_1,\ldots,u_n)\in \R^n\times\R^n\times\R^n, j=1,\ldots,n.
	\end{align*}
\end{enumerate}
\subsection{A multi-objective variational problem of Herglotz type}

Consider an $n$-tuple $(\xi_1,\ldots,\xi_n)\in\prod \Gamma^t_{x,y}$ and the following ordinary differential equation in the sense of Carath\'eodory
\begin{equation}\label{eq_Caratheodory_systems}
	\left\{
	\begin{matrix}
		\dot{u}_1(s)=L^1(\xi_1(s),\dot{\xi}_1(s),u_1(s),\ldots,u_n(s)),\\
		\dot{u}_2(s)=L^2(\xi_2(s),\dot{\xi}_2(s),u_1(s),\ldots,u_n(s)),\\
		\cdots\cdots\cdots\cdots\cdots\cdots\cdots\cdots\cdots\cdots \\
		\dot{u}_n(s)=L^n(\xi_n(s),\dot{\xi}_n(s),u_1(s),\ldots,u_n(s)),
	\end{matrix}
	\right.\qquad\qquad s\in[0,t],
\end{equation}
with initial conditions $u_i(0)=a_i\in\R$, $i=1,\ldots,n$. Equation \eqref{eq_Caratheodory_systems} admits a unique solution $(u_i,\ldots,u_n)$ where
\begin{align*}
	u_i(s)=u_i(\xi_1,\ldots,\xi_n,s),\quad s\in[0,t],\ i=1,\ldots n.
\end{align*}

We denote $L^i_0(x,v)=L^i(x,v,0,\ldots,0)$. It is clear that $L^i_0(x,v)$ is a Tonelli Lagrangian. Each entry of Carath\'eodory equation \eqref{eq_Caratheodory_systems} has the form
\begin{align*}
	\dot{u}_i(s)=&\,L^i(\xi_i(s),\dot{\xi}_i(s),u_1(s),\ldots,u_n(s))\\
	=&\,L^i_0(\xi_i(s),\dot{\xi}_i(s))+\sum^n_{j=1}\widehat{L^i_{u_j}}(s)\cdot u_j(s),
\end{align*}
where
\begin{align*}
	\widehat{L^i_{u_j}}(s)=\int^1_0L^i_{u_j}(\xi_i(s),\dot{\xi}_i(s),\lambda u_1(s),\ldots,\lambda u_n(s))\ d\lambda.
\end{align*}
Let
\begin{align*}
	\bL_0(\bx,\bv)=\begin{pmatrix}
		L^1_0(x_1,v_1)\\
		L^2_0(x_2,v_2)\\
		\vdots\\
		L^n_0(x_n,v_n)
	\end{pmatrix},\quad
	\widehat{\bL_{\bu}}=
	\begin{pmatrix}
		\widehat{L^1_{u_1}},\cdots,\widehat{L^1_{u_n}}\\
		\widehat{L^1_{u_2}},\cdots,\widehat{L^2_{u_n}}\\
		\vdots\\
		\widehat{L^1_{u_n}},\cdots,\widehat{L^n_{u_n}}
	\end{pmatrix},\quad
	\bxi=\begin{pmatrix}
		\xi_1\\
		\xi_2\\
		\vdots\\
		\xi_n
	\end{pmatrix}.
\end{align*}
Therefore \eqref{eq_Caratheodory_systems} becomes
\begin{align*}
	\dot{\bu}=\bL_0(\bxi,\dot{\bxi})+\widehat{\bL_{\bu}}\cdot\bu,\quad \bu(0)=\ba.
\end{align*}
By solving this equation we have
\begin{equation}\label{eq:Cara_compact}
	\bu(t)=e^{\int^t_0\widehat{\bL_{\bu}}\ ds}\cdot\ba+\int^t_0e^{\int^t_s\widehat{\bL_{\bu}}\ ds}\cdot\bL_0(\bxi,\dot{\bxi})\ ds,\quad \bu(0)=\ba.
\end{equation}
We remark the matrix $\widehat{\bL_{\bu}}$ is implicitly determined by Carath\'eodory equation \eqref{eq_Caratheodory_systems} in general. But the problem becomes much easier when $\bL(\bx,\bv,\bu)$ is linear in $\bu$.

For any $i=1,\ldots,n$ we define
\begin{align*}
	J^*_i(\xi_1,\ldots,\xi_n)=\int^t_0L^i(\xi_i(s),\dot{\xi}_i(s),u_1(s),\ldots,u_n(s))\ ds,
\end{align*}
where $u_1(s),\ldots,u_n(s)$ are uniquely determined by \eqref{eq_Caratheodory_systems}. Our task is to find a way to understand the multi-objects optimization problem with respect to the functional $\{J^*_i\}$ for $\xi_1,\ldots,\xi_n\in\Gamma^t_{x,y}$. 

Keep in mind the dynamic programming principle for Bolza problem. So, it is natural to use a reduction of the functionals $\{J^*_i\}$ by setting
\begin{align*}
	J_i(\xi)=J^*_i(\xi,\ldots,\xi),\qquad \xi\in\Gamma^t_{x,y}.
\end{align*}
Now, our problem becomes minimizing the functional 
\begin{equation}\label{eq:cov}
	J_i(\xi)\quad \text{for}\quad \xi\in\Gamma^t_{x,y},
\end{equation}
for each $i=1,\ldots,n$, or equivalently,
\begin{align*}
	\text{minimize}\quad u_i(\xi,t)
\end{align*}
where $u_i$ is determined by
\begin{equation}\label{eq:cara_3}
	\left\{
	\begin{matrix}
		\dot{u}_1(s)=L^1(\xi(s),\dot{\xi}(s),u_1(s),\ldots,u_n(s)),\\
		\dot{u}_2(s)=L^2(\xi(s),\dot{\xi}(s),u_1(s),\ldots,u_n(s)),\\
		\cdots\cdots\cdots\cdots\cdots\cdots\cdots\cdots\cdots\cdots \\
		\dot{u}_n(s)=L^n(\xi(s),\dot{\xi}(s),u_1(s),\ldots,u_n(s)),
	\end{matrix}
	\right.\qquad\qquad s\in[0,t],
\end{equation}
with initial data $u_i(0)=a_i$, $i=1,\ldots,n$.

\subsection{Linearly coupled systems}

The existence and regularity of the minimizer of problem \eqref{eq:cov} are absolutely nontrivial in general, not only because of technical difficulty. We will handle this general problem of contact type in the forthcoming paper.

In the current paper, we will deal with a widely studied system with Lagrangians as in \eqref{eq:L_intro}, i.e.,
\begin{equation}\label{eq:L}
	\bL(\bx,\bv)-\bA\bu.
\end{equation}
where $\bA=\{a^i_j\}$ is an $n\times n$-matrix with real entries satisfying conditions (C1) and (C2) in the introduction.

For the Lagrangians in the form \eqref{eq:L}, by solving Carath\'eodory equation \eqref{eq:cara_3} in a compact form we have that
\begin{equation}\label{eq:bu}
	\bu(\bxi,t)=e^{-\bA t}\cdot\ba+\int^t_0e^{\bA(s-t)}\cdot\bL(\bxi,\dot{\bxi})\ ds,\quad \bu(\bxi,0)=\ba.
\end{equation}
We denote $e^{-\bA s}=\{b^i_j(s)\}$, $e^{\bA s}=\{c^i_j(s)\}$ and $e^{\bA (s-t)}=\{d^i_j(s)\}$. Thus, for each $i$,
\begin{align*}
	J_i(\xi)=u_i(\xi,\ldots,\xi,t)=\sum^n_{j=1}b^i_j(t)a_j+\int^t_0\mathbb{L}^i(s,\xi,\dot{\xi})\ ds.
\end{align*}
where $\mathbb{L}^i(s,x,v)$ is a time-dependent Lagrangian defined by
\begin{equation}\label{eq:Lagrangian_t}
	\mathbb{L}^i(s,x,v)=\sum^n_{k,j=1}b^i_k(t)c^k_j(s)L^j(x,v)=\sum^n_{j=1}d^i_j(s)L^j(x,v).
\end{equation}

\begin{Rem}\label{rem:conditions}
We have to verify some points to ensure this variational problem can deduce the original one and we can solve it. Under conditions (C1) and (C2) we ensure that all the entries of $e^{-\bA t}$ are postive for all $t>0$.
\begin{enumerate}[--]
	\item Under conditions (C1) and (C2) we also ensure that each $\mathbb{L}^i$ is a time-dependent Tonelli Lagrangian. In particular, $\frac d{dt}\mathbb{L}^i(t,x,v)$ is dominated by $\mathbb{L}^i(t,x,v)$. So, we can have full regularity of the minimizers.
	\item Condition (C2) also ensures that our Carath\'eodory system \eqref{eq:cara_3} cannot be split into sub-systems. Indeed, this condition is not necessary and it just guarantees all the coefficients of the weighted Lagrangians do not vanish.
	\item For arbitrary coupling matrix $\bA$, $e^{\bA t}\sim I$ for $t\ll1$. Then, we can deal with the problem in a small time interval for the family $\{L^i\}$ with certain uniform superlinear growth and uniform convexity  similarly (see Section 5).
\end{enumerate}
\end{Rem}

\subsubsection{Herglotz equation}

In view of Remark \ref{rem:conditions}, we always suppose conditions (C1) and (C2). 

\begin{The}\label{thm:Herglotz}
Suppose conditions \mbox{\rm (C1)} and \mbox{\rm (C2)}. Then, for each $i$, $u_i(\xi,t)$ admits a solution and each minimizer $\xi_i$ of $u_i(\xi,t)$ is of class $C^2$. Moreover, the system of extremals $(\xi_1,\ldots,\xi_n)$ satisfies the following Euler-Lagrange equations
\begin{align*}
	\frac d{ds}\mathbb{L}_v^i(s,\xi_i,\dot{\xi}_i)=\mathbb{L}_x^i(s,\xi_i,\dot{\xi}_i),\qquad i=1,\ldots,n.
\end{align*}
More precisely, the Euler-Lagrange equations have the form of Herglotz
\begin{equation}\label{eq:Herglotz}
	\frac d{ds}L^i_v(\xi_i(s),\dot{\xi}_i(s))=L^i_x(\xi_i(s),\dot{\xi}_i(s))-\sum^n_{j=1}a^i_jL^j_v(\xi_i(s),\dot{\xi}_i(s)),\qquad i=1,\ldots,n.
\end{equation}
\end{The}


\begin{proof}
Indeed, the associated Euler-Lagrange equations have the form
\begin{align*}
	\sum^n_{j,k=1}b^i_k(t)\frac d{ds}c^k_j(s)L^j_v(\xi_i(s),\dot{\xi}_i(s))=\sum^n_{j,k=1}b^i_k(t)c^k_j(s)L^j_x(\xi_i(s),\dot{\xi}_i(s)),\qquad i=1,\ldots,n.
\end{align*}
Multiplying each $c^m_i(t)$ to the $m$th-line in the equality above and summing up, we obtain
\begin{equation}\label{eq:EL1}
	\sum^n_{j=1}\frac d{ds}c^i_j(s)L^j_v(\xi_i(s),\dot{\xi}_i(s))=\sum^n_{j=1}c^i_j(s)L^j_x(\xi_i(s),\dot{\xi}_i(s)),\qquad i=1,\ldots,n.
\end{equation}
Recall $(c^i_j(s))=e^{\bA s}$ and $((c^i_j)'(s))=e^{\bA s}A$. That is
\begin{align*}
	(c^i_j)'(s)=\sum^n_{k=1}c^i_k(s)a^k_j.
\end{align*}
It follows
\begin{align*}
	\sum^n_{j=1}\frac d{ds}c^i_j(s)L^j_v(\xi_i(s),\dot{\xi}_i(s))=&\,\sum^n_{j=1}c^i_j(s)\frac d{ds}L^j_v(\xi_i(s),\dot{\xi}_i(s))+\sum^n_{j=1}(c^i_j(s))'L^j_v(\xi_i(s),\dot{\xi}_i(s))\\
	=&\,\sum^n_{j=1}c^i_j(s)\frac d{ds}L^j_v(\xi_i(s),\dot{\xi}_i(s))+\sum^n_{j,k=1}c^i_k(s)a^k_jL^j_v(\xi_i(s),\dot{\xi}_i(s)).
\end{align*}
Thus \eqref{eq:EL1} becomes
\begin{equation}\label{eq:EL2}
	\sum^n_{j=1}c^i_j(s)\frac d{ds}L^j_v(\xi_i(s),\dot{\xi}_i(s))+\sum^n_{j,k=1}c^i_k(s)a^k_jL^j_v(\xi_i(s),\dot{\xi}_i(s))=\sum^n_{j=1}c^i_j(s)L^j_x(\xi_i(s),\dot{\xi}_i(s)).
\end{equation}
This leads to \eqref{eq:Herglotz} by multiplying each $b^m_i(s)$ to $m$th-line in \eqref{eq:EL2} and summing up. This completes the proof.
\end{proof}

\subsubsection{Inf-convolution of convex functions}

To understand the Hamiltonians $\{\mathbb{H}^i\}$'s associated to $\{\mathbb{L}^i\}$'s, we need recall more facts from convex analysis (see, for instance, \cite{Hiriart-Urruty_Lemarechal_book2001}). 

\begin{Pro}\label{eq:convex}
Suppose $f_1,\ldots,f_k,g$ are $C^2$, strictly convex functions with superlinear growth. Let $\lambda_1,\ldots,\lambda_k,\lambda>0$ and $f=\sum^k_{j=1}\lambda_jf_j$.
\begin{enumerate}[\rm (a)]
	\item $f$ is also a $C^2$ strictly convex function with superlinear growth as well as $f^*$, the convex dual of $f$.
	\item $(f^*)^*=f$.
	\item $(\lambda g)^*(y)=\lambda g^*(y/\lambda)$.
	\item $(f_1+\cdots+f_k)^*(y)=(f_1^*\uplus\cdots\uplus f_k^*)(y)$
	where $\uplus$ stands for the inf-convolution of convex functions defined by
	\begin{equation}\label{eq:inf_conv}
		(f_1\uplus\cdots\uplus f_k)(x)=\inf\{f_1(x_1)+\cdots+f_k(x_k): x_1+\cdots+x_k=x\}.
	\end{equation}
	\item $\nabla(f_1\uplus\cdots\uplus f_k))(x)=\nabla f_1(x_1^*)=\cdots=\nabla f_k(x_k^*)$ where $(x_1^*,\ldots,x_k^*)$ is the unique minimizer in \eqref{eq:inf_conv} such that $x_1^*+\cdots+x_k^*=x$.
\end{enumerate} 
\end{Pro}

We introduce the Hamiltonians with respect to the Lagrangians $\{L^i\}$ in \eqref{eq:L}. Let for each $i$
\begin{align*}
	H^i(x,p)=\sup_{v\in\R^n}\{p\cdot v-L^i(x,v)\}.
\end{align*}
Due to Proposition \ref{eq:convex}, for each $i$, the  Hamiltonian $\mathbb{H}^i$ associated to $\mathbb{L}^i$ in \eqref{eq:Lagrangian_t} is
\begin{equation}\label{eq:Hamiltonian_t}
	\mathbb{H}^i(s,x,p)=\inf\bigg\{\sum^n_{j=1}d^i_j(s)H^j(x,q_j/d^i_j(s)): q_1,\ldots,q_n\in\R^n, \sum^n_{j=1}q_j=p\bigg\}.
\end{equation}
In the definition of $\mathbb{H}^i$, we conclude the infimum can be achieved since $\sum^n_{j=1}c^i_j(s)H^j(x,q_j/c^i_j(s))$ as a function of $(q_1,\ldots,q_n)$ is superlinear, and the minimizer is unique because $\mathbb{H}^i$ is smooth. Given $(s,x,p)$, there exists a unique $(q^1_i(s),\ldots,q^n_i(s))$ such that
\begin{equation}\label{eq:property_H}
	\begin{split}
		p=\sum^n_{j=1}q^i_j(s),\quad \mathbb{H}^i(s,x,p)=\sum^n_{j=1}d^i_j(s)H^j(x,q^i_j(s)/d^i_j(s)),\\
	\mathbb{H}^i_p(s,x,p)=H^1_p(x,q^1_i(s)/d^i_1(s))=\cdots=H^n_p(x,q^n_i(s)/d^i_n(s))\\
	\mathbb{H}^i_x(s,x,p)=\sum^n_{j=1}d^i_j(s)H^j_x(x,q^j_i(s)/d^i_j(s)).
	\end{split}
\end{equation}

\subsubsection{Lie equation and characteristic systems}

We also write down the Hamiltonian equations with respect to $\mathbb{H}^i$ ($i=1,\ldots,n$) as
\begin{align*}
	\begin{cases}
		\dot{\xi}_i(s)=\mathbb{H}^i_p(s,\xi_i(s),p_i(s)),\\
		\dot{p}_i(s)=-\mathbb{H}^i_x(s,\xi_i(s),p_i(s)).
	\end{cases}
\end{align*}

\begin{Lem}\label{lem:bH}
Let $p_i(s):=\mathbb{L}^i_v(s,\xi_i(s),\dot{\xi}_i(s))$ be the dual arc. Then
\begin{align*}
	p_i(s)=\sum^n_{j=1}d^i_j(s)p^j_i(s),\qquad i=1,\cdots,n,
\end{align*}
where $p^j_i(s)=L^j_v(\xi_i(s),\dot{\xi}_i(s))$ for all $i,j=1,\ldots,n$. Moreover, we have
\begin{align*}
	\mathbb{H}^i(s,\xi_i(s),p_i(s))=\sum^n_{j=1}d^i_j(s)H^j(\xi_i(s),p^j_i(s)),\quad i=1,\cdots,n.
\end{align*}
\end{Lem}

\begin{Rem}
The curves $\{d^i_j(s)p^j_i(s)\}_{j=1}^n$ are exactly the minimizers for $\mathbb{H}^i(s,\xi(s),p_i(s))$ in the inf-convolution representation \eqref{eq:Hamiltonian_t} by \eqref{eq:property_H}, for each $i$.
\end{Rem}

\begin{proof}
From the definition of $\mathbb{L}^i$, we obtain
\begin{align*}
	p_i(s)=\mathbb{L}^i_v(s,\xi_i(s),\dot{\xi}_i(s))=\sum^n_{j=1}d^i_j(s)L^j_v(\xi_i(s),\dot{\xi}_i(s))=\sum^n_{j=1}d^i_j(s)p^j_i(s),
\end{align*}
Therefore
\begin{align*}
	&\,\mathbb{H}^i(s,\xi_i(s),\mathbb{L}^i_v(s,\xi_i(s),\dot{\xi}_i(s)))\\
	=&\,\mathbb{L}^i_v(s,\xi_i(s),\dot{\xi}_i(s)))\cdot\dot{\xi}_i(s)-\mathbb{L}^i(s,\xi_i(s),\dot{\xi}_i(s)))\\
	=&\,\sum^n_{j=1}d^i_j(s)\big\{L^j_v(\xi_i(s),\dot{\xi}_i(s))\cdot\dot{\xi}_i(s)-L^j(\xi_i(s),\dot{\xi}_i(s))\big\}\\
	=&\,\sum^n_{j=1}d^i_j(s)H^j(\xi_i(s),p^j_i(s)).
\end{align*}
This completes our proof.
\end{proof}

By the observation above the associated Hamiltonian systems of \eqref{eq:Herglotz} are more complicated. Similar to Theorem \ref{thm:Herglotz} we have the following characteristic system in Hamiltonian form.

\begin{The}\label{thm:Lie}
Let $(\xi_1,\cdots,\xi_n)$ be the minimizers of \eqref{eq:cov} for $\bL$ as in \eqref{eq:L}. For each $i,j=1,\cdots,n$, let $p^j_i(s)=L^j_v(\xi_i(s),\dot{\xi}_i(s))$. Then, the curves $\{\xi_i\}$ and $\{p^i_j\}$ satisfy
\begin{equation}\label{eq:Lie1}
	\begin{split}
		\dot{\xi}_i(s)=H_p^j(\xi_j(s),p^j_i(s)),\qquad i,j=1,\ldots,n,\\
	\dot{p}^i_i(s)=-H^i_x(\xi_i(s),p^i_i(s))+\sum^n_{j=1}a^i_jp^j_i(s),\qquad i=1,\ldots,n.\\
	\end{split}
\end{equation}
\end{The}

To obtain a clear picture of the characteristic system, we also need consider the $\bu$-curves. Comparing to the scalar case, the $\bu$-curves are much complicated as well as $\bp$-curves explained in the last subsection. Suppose $(\xi_1,\cdots,\xi_n)$ are minimizers of \eqref{eq:cov} for $\bL$ as in \eqref{eq:L}. Then, Carath\'eodory equation \eqref{eq:cara_3} defines for each pair $(i,j)$ a curve $u^j_i$ and such $n\times n$ number of curves $\{u^j_i\}$ satisfy
\begin{align*}
	\dot{u}^j_i=L^j(\xi_i,\dot{\xi}_i)+\sum^n_{j=1}a^j_ku^k_i,\qquad i,j=1,\ldots,n.
\end{align*}
In the Hamiltonian formalism, they have the form
\begin{align*}
	\dot{u}^j_i=p^j_i\cdot H^j_p(\xi_i,p^j_i)-H(\xi_i,p^j_i)+\sum^n_{j=1}a^j_ku^k_i,
\end{align*}
where $\{p^j_i\}$ are determined by Theorem \eqref{thm:Lie}. Because of linear coupling, the $\bu$-factor of the characteristic systems is determined by the $(\bxi,\bp)$-factor. But it is not the case in general if we consider general nonlinear coupling.

\section{Weakly coupled system of Hamilton-Jacobi equations}


Let $\phi_i\in BUC(\R^n,\R)$, $i=1,\ldots,n$. For any $(t,x)\in(0,\infty)\times\R^n$ we define the set of accessible arcs by
\begin{align*}
	\mathcal{A}_{t,x}=\{\xi\in W^{1,1}([0,t],\R^n): \xi(t)=x\}.
\end{align*}
Consider the following Bolza problem for systems: for $t>0$ and $x\in\R^n$, 
\begin{equation}\label{eq:Bolza}
	u^i(t,x)=\inf_{\xi\in\mathcal{A}_{t,x}}\left\{\sum^n_{j=1}b^i_j(t)\phi_j(\xi(0))+\int^t_0\sum^n_{j=1}d^i_j(s)L^j(\xi,\dot{\xi})\ ds\right\}
\end{equation}

\begin{The}\label{thm:viscosity}
Each $u^i$ defined in \eqref{eq:Bolza} is locally semiconcave, $i=1,\ldots,n$, and $(u^1,\ldots,u^n)$ satisfies the following weakly coupled systems of Hamilton-Jacobi equations
\begin{equation}\tag{HJ$_S-$}\label{eq:HJS}
	\begin{cases}
		D_tu^1(t,x)+H^1(x,D_xu^1(t,x))+\sum^n_{j=1}a^1_ju^j(t,x)=0,\\
		\cdotss\\
		D_tu^n(t,x)+H^n(x,D_xu^n(t,x))+\sum^n_{j=1}a^n_ju^j(t,x)=0
	\end{cases}
\end{equation}
on $(0,+\infty)\times\R^n$ in the sense of viscosity. Moreover, if $(t,x)\in(0,\infty)\times\R^n$ is a point of differentiability for all $u_i$'s, then for each $i$ we have
\begin{align*}
	D_xu^i(t,x)=&\,L^i_v(\xi_i(t),\dot{\xi}_i(t))\\
	D_tu^i(t,x)=&\,-H^i(\xi_i(t),L^i_v(\xi_i(t),\dot{\xi}_i(t))-\sum^n_{k=1}a^i_ku^k(t,x),
\end{align*}
where the $C^2$ curve $\xi_i\in\mathcal{A}_{t,x}$ is the unique minimizer for $u^i(t,x)$, $i=1,\ldots,n$. 
\end{The}

\begin{proof}
One can rewrite \eqref{eq:Bolza} as
\begin{equation}\label{eq:marginal_function}
	u^i(t,x)=\inf_{\xi\in\mathcal{A}_{t,x}}\left\{\sum^n_{j=1}b^i_j(t)\phi_j(z)+A_{0,t}(z,x)\right\},
\end{equation}
where 
\begin{align*}
	A_{0,t}(x,y)=\inf_{\xi\in\Gamma^t_{x,y}}\int^t_0\mathbb{L}^i(s,\xi(s),\dot{\xi}(s))\ ds,\qquad x,y\in\R^n, t>0.
\end{align*}
Invoking Lemma 3.1 in \cite{Cannarsa_Cheng3}, the infimum in \eqref{eq:marginal_function} can be achieved. Observe we represent each $u^i$ as a family of functions $\sum^n_{j=1}b^i_j(t)\phi_j(z)+A_{0,t}(z,x)$ ($z$ as a parameter). We conclude $u^i$ is locally semiconcave since for each $z$ the function $(t,x)\mapsto \sum^n_{j=1}b^i_j(t)\phi_j(z)+A_{0,t}(z,x)$ is uniformly semiconcave on any compact subset of $(0,\infty)\times\R^n$ by Theorem 3.4.4 in \cite{Cannarsa_Sinestrari_book}.

Now, suppose $(t,x)\in(0,\infty)\times\R^n$ is a point of differentiability for all $u_i$'s. Applying Theorem 3.4.4 in \cite{Cannarsa_Sinestrari_book} again we know there exists a unique $z$ such that
\begin{align*}
	u^i(t,x)=\sum^n_{j=1}b^i_j(t)\phi_j(z)+A_{0,t}(z,x),
\end{align*}
and
\begin{align*}
	D_tu^i(t,x)=&\,\sum^n_{j=1}(b^i_j(t))'\phi_j(z)-\mathbb{H}^i(t,\xi_i(t),\mathbb{L}^i_v(t,\xi_i(t),\dot{\xi}_i(t)))+\sum^n_{j=1}(b^i_j(t))'\int^t_0\sum^n_{k=1}c^j_k(s)L^k(\xi_i,\dot{\xi}_i)\ ds\\
	D_xu^i(t,x)=&\,\sum^n_{j,k=1}b^i_j(t)c^j_k(t)L^k_v(\xi_i(t),\dot{\xi}_i(t)),
\end{align*}
where $\xi^t_{z,x}$ is the unique minimal curve for $A_{0,t}(z,x)$. Observe $\frac d{dt}e^{-\bA t}=-\bA\cdot e^{-\bA t}$, i.e.
\begin{equation}\label{eq:deAt}
	(b^i_j(t))'=-\sum^n_{k=1}a^i_kb^k_j(t),\qquad i,j=1,\ldots,n.
\end{equation}
Thus,
\begin{align*}
	&\,\sum^n_{j=1}(b^i_j(t))'\phi_j(z)+\sum^n_{j=1}(b^i_j(t))'\int^t_0\sum^n_{k=1}c^j_k(s)L^k(\xi_i,\dot{\xi}_i)\ ds\\
	=&\,-\sum^n_{j=1}a^i_j\bigg\{\sum^n_{k=1}b^j_k(t)\phi_k(z)+\sum^n_{j=1}b^j_k(t)\int^t_0\sum^n_{k=1}c^k_m(s)L^m(\xi_i,\dot{\xi}_i)\ ds\bigg\}\\
	=&\,-\sum^n_{j=1}a^i_ju^j(t,x)
\end{align*}
Thus, by Lemma \ref{lem:bH} we conclude
\begin{equation}\label{eq:Dx}
	D_xu^i(t,x)=\sum^n_{j,k=1}b^i_k(t)c^k_j(t)L^j_v(\xi_i(t),\dot{\xi}_i(t))=L^i_v(\xi_i(t),\dot{\xi}_i(t))
\end{equation}
and
\begin{equation}\label{eq:Dt}
	\begin{split}
		D_tu^i(t,x)=&\,-\sum^n_{j=1}a^i_ju^j(t,x)-\sum^n_{k,j=1}b^i_k(t)c^k_j(t)H^k(\xi_k(t),L^j_v(\xi_k(t),\dot{\xi}_k(t)))\\
		=&\,-\sum^n_{j=1}a^i_ju^j(t,x)-H^i(\xi_i(t),L^i_v(\xi_i(t),\dot{\xi}_i(t))),
	\end{split}
\end{equation}
since $\sum^n_{k=1}b^i_k(t)c^k_j(t)=\delta^i_j$, the Kronecker symbol, $i,j=1,\ldots,n$. Recall that $u^i$'s are all locallly semiconcave. Then, the combination of \eqref{eq:Dx} and \eqref{eq:Dt} together with Proposition 5.3.1 in \cite{Cannarsa_Sinestrari_book} completes the proof.
\end{proof}

\section{Lax-Oleinik evolution of weakly coupled systems}

Fix $x,y\in\R^n$, $t>0$ and $\ba=(a_1,\ldots,a_n)\in\R^n$. For any $\xi\in\Gamma^t_{x,y}$ we will consider two kinds of problem of the following Carath\'eodory system
\begin{equation}\label{eq:cara4}
	\left\{
	\begin{matrix}
		\dot{u}_1(\xi,s)=L^1(\xi(s),\dot{\xi}(s))-\sum_{j=1}^na^1_ju_j(\xi,s),\\
		\dot{u}_2(\xi,s)=L^2(\xi(s),\dot{\xi}(s))-\sum_{j=1}^na^2_ju_j(\xi,s),\\
		\cdots\cdots\cdots\cdots\cdots\cdots\cdots\cdots\cdots\cdots \\
		\dot{u}_n(\xi,s)=L^n(\xi(s),\dot{\xi}(s))-\sum_{j=1}^na^n_ju_j(\xi,s),
	\end{matrix}
	\right.\qquad\qquad s\in[0,t].
\end{equation}


\begin{defn}
For any $x,y\in\R^n$, $t>0$ and $\ba=(a_1,\ldots,a_n)\in\R^n$ we define
\begin{equation}\label{eq:fund_sol-}
	h_i(\bL-\bA\bu,t,x,y,\ba)=\inf_{\xi}\int^t_0\big\{L^i(\xi(s),\dot{\xi}(s))-\sum_{j=1}^na^i_ju_j(\xi,s)\big\}\ ds,
\end{equation}
where $\xi\in\Gamma^t_{x,y}$ and $\{u_i(\xi,\cdot)\}$ are determined by \eqref{eq:cara4} under initial conditions $u_i(\xi,0)=a_i$, $i=1,\ldots,n$. The functions $\{h_i(\bL,t,x,y,\ba)\}$ are called the \emph{negative type fundamental solutions} of the weakly coupled systems of Hamilton-Jacobi equations.
\end{defn}

Similar to the problem in \eqref{eq:fund_sol-} we can also consider a terminal conditions problem.

\begin{defn}
For any $x,y\in\R^n$, $t>0$ and $\ba=(a_1,\ldots,a_n)\in\R^n$ we define
\begin{equation}\label{eq:fund_sol+}
	\breve{h}_i(\bL-\bA\bu,t,x,y,\ba)=\inf_{\xi}\int^t_0\big\{L^i(\xi(s),\dot{\xi}(s))-\sum_{j=1}^na^i_ju_j(\xi,s)\big\}\ ds,
\end{equation}
where $\xi\in\Gamma^t_{x,y}$ and $\{u_i(\xi,\cdot)\}$ are determined by \eqref{eq:cara4} under terminal conditions $u_i(\xi,t)=a_i$, $i=1,\ldots,n$. The functions $\{\breve{h}_i(\bL,t,x,y,\ba)\}$ are called the \emph{positive type fundamental solutions} of the weakly coupled systems of Hamilton-Jacobi equations.
\end{defn}

The existence and regularity of the minimizers for $h_i(\bL,t,x,y,\ba)$ have been obtained in Section 2 by reducing the problem to the classical theorem of Tonelli. In fact, we can deal with $\breve{h}_i(\bL,t,x,y,\ba)$ in a similar way. One can also find the direct relations between $h_i$ and $\breve{h}_i$ which is useful for our further analysis. For each $i$ define $\breve{L}^i(x,v)=L^i(x,-v)$ and denote by $\breve{\bL}$  in vectorial form of the Lagrangians $\{\breve{L}^i\}$.


\begin{Lem}\label{lem:equiv}
For any $x,y\in\R^n$, $\ba\in\R^n$ and $t>0$, 
\begin{align*}
	\breve{h}_i(\bL-\bA\bu,t,x,y,\ba)=h_i(\breve{\bL}+\bA\bu,t,y,x,-\ba).
\end{align*}	
Moreover, for any $\xi\in\Gamma^t_{y,x}$ and $\{u_i(\xi,\cdot)\}$ determined by \eqref{eq:cara4} with respect to $\breve{\bL}+\bA\bu$, with initial conditions $u_i(\xi,0)=-a_i$, $i=1,\ldots,n$, we define 
	\begin{align*}
		\eta(s)=\xi(t-s),\quad u_i(\eta,s)=-u_i(\xi,t-s),\quad s\in[0,t],\ i=1,\ldots,n.
	\end{align*}
	Then, for each $i$, $\xi\in\Gamma^t_{y,x}$ is a minimal curve for $h_i(\breve{\bL}+\bA\bu,t,y,x,-\ba)$ with $\{u_i(\xi,\cdot)\}$ determined by \eqref{eq:cara4} with initial conditions $u_i(\xi,0)=-a_i$ if and only if $\eta\in\Gamma^t_{x,y}$ is a minimal curve for $\breve{h}_i(\bL-\bA\bu,t,x,y,\ba)$ with $\{u_i(\eta,\cdot)\}$ determined by \eqref{eq:cara4} with initial conditions $u_i(\eta,t)=a_i$, $i=1,\ldots,n$.
\end{Lem}

\begin{proof}
	Fix $x,y\in\R^n$, $t>0$. Let $\xi\in\Gamma^t_{y,x}$ and let $\{u_i(\xi,\cdot)\}$ be determined by \eqref{eq:cara4} with respect to $\breve{\bL}+\bA\bu$, with initial conditions $u_i(\xi,0)=-a_i$, $i=1,\ldots,n$, i.e.,
	\begin{align*}
		h_i(\breve{\bL}+\bA\bu,t,x,y,-\ba)=\inf_{\xi}\int^t_0\big\{L^i(\xi(s),-\dot{\xi}(s))+\sum_{j=1}^na^i_ju_j(\xi,s)\big\}\ ds
	\end{align*}
	Notice that $\dot{\eta}(s)=-\dot{\xi}(t-s)$, $u_i(\eta,s)=-u_i(\xi,t-s)$  with $\{u_i(\eta,\cdot)\}$ determined by \eqref{eq:cara4} with initial conditions $u_i(\eta,t)=a_i$, $i=1,\ldots,n$. Therefore
	\begin{align*}
		\breve{h}_i(\bL-\bA\bu,t,x,y,\ba)\leqslant&\,\int^t_0\big\{L^i(\eta(s),\dot{\eta}(s))-\sum_{j=1}^na^i_ju_j(\eta,s)\big\}\ ds\\
		=&\,\int^t_0\big\{L^i(\eta(t-\tau),\dot{\eta}(t-\tau))-\sum_{j=1}^na^i_ju_j(\eta,t-\tau)\big\}\ d\tau\\
		=&\,\int^t_0\big\{L^i(\xi(\tau),-\dot{\xi}(\tau))+\sum_{j=1}^na^i_ju_j(\xi,\tau)\big\}\ d\tau\\
		=&\,h_i(\breve{\bL}+\bA\bu,t,y,x,-\ba).
	\end{align*}
	The opposite inequality can be obtained similarly.
\end{proof}

Now we can define the Lax-Oleinik evolution.

\begin{defn}
For any $\phi_i:\R^n\to\R$, $i=1,\ldots,n$, and $t>0$, we define
\begin{equation}
	\begin{split}
		T^i_{t}\phi_i(x)=&\,\inf_{y\in\R^n}\{\phi_i(y)+h_i(\bL-\bA\bu,t,x,y,\phi_1(y),\ldots,\phi_n(y))\},\\
		\breve{T}^i_{t}\phi_i(x)=&\,\sup_{y\in\R^n}\{\phi_i(y)-\breve{h}_i(\bL-\bA\bu,t,y,x,\phi_1(y),\ldots,\phi_n(y)).\}
	\end{split}
\end{equation}
In a compact form, we set $\bphi=(\phi_1,\cdots,\phi_n)$ and write 
\begin{align*}
	\mathbb{T}_t\bphi(x)=
	\begin{pmatrix}
		T^1_t\phi_1(x)\\
		\cdots\\
		T^n_t\phi_n(x)
	\end{pmatrix},\qquad
	\breve{\mathbb{T}}_t\bphi(x)=
	\begin{pmatrix}
		\breve{T}^1_t\phi_1(x)\\
		\cdots\\
		\breve{T}^n_t\phi_n(x)
	\end{pmatrix}.
\end{align*}
We call $\mathbb{T}_t$ and $\breve{\mathbb{T}}_t$ the \emph{negative} and \emph{positive type Lax-Oleinik evolution} respectively.
\end{defn}

As an easy consequence of Lemma \ref{lem:equiv} and the definition of $\mathbb{T}_t$ and $\breve{\mathbb{T}}_t$ we have the following consequence.

\begin{Cor}
For any $\phi_i:\R^n\to\R$, $i=1,\ldots,n$, and $t>0$, we have
\begin{align*}
	-\breve{T}^i_{t}\phi_i=T^i_t(-\phi_i).
\end{align*}
\end{Cor}

In Section 3, we have already verified that, if all $\phi_i$'s are bounded and uniformly continuous, then $u^i=T^i_t\phi_i$, $i=1,\cdots,n$, are indeed the viscosity solutions of weakly coupled systems of Hamilton-Jacobi equations \eqref{eq:HJS}. This implies the following result with wide interests in the literature.

\begin{The}
Suppose $\phi_1,\ldots,\phi_n\in BUC(\R^n,\R)$. Then $u_i(t,x)=-\breve{T}^i_{t}\phi_i(x)$, $i=1,\ldots,n$, are solutions of the following weakly coupled systems of Hamilton-Jacobi equations
\begin{equation}\tag{HJ$S_+$}\label{eq:HJS+}
	\begin{cases}
		D_tu_1(t,x)+H^1(x,-D_xu_1(t,x))-\sum^n_{j=1}a^1_ju_j(t,x)=0,\\
		    \cdotss\\
		    D_tu_n(t,x)+H^n(x,-D_xu_n(t,x))-\sum^n_{j=1}a^n_ju_j(t,x)=0,
	\end{cases}
\end{equation}
on $(0,\infty)\times\R^n$ with initial conditions $u_i(0,x)=-\phi_i(x)$, in the sense of viscosity. 
\end{The}

\section{Concluding Remark}

The method developed in this paper can also be applied for arbitrary real coupling matrix $\bA$. For arbitrary coupling matrix $A$, $e^{At}\sim I$ for $t\ll1$. Then, we can deal with the problem in a small time interval. We impose a condition as follows.
\begin{enumerate}[(C1)]\setliststart{3}
    \item There exist two superlinear functions $\theta_0,\theta_1:[0,\infty)\to[0,\infty)$ and $c_0>0$ such that
    \begin{enumerate}[(i)]
        \item $\theta_1(|v|)\geqslant L^i(x,v)\geqslant\theta_0(|v|)-c_0$ for all $(x,v)\in\R^n\times\R^n$ and $i=1,\ldots,n$, and there exists $C_1>0$ such that $\theta_1(\nu)\leqslant C_1(1+\theta_0(\nu))$ for all $\nu\in[0,\infty)$.
        \item There exists $C_2>0$ such that $\sum^n_{j=1}D^2L^j(x,\cdot)\leqslant C_2D^2L^i(x,\cdot)$ in the sense of distribution for all $x\in\R^n$ and $i=1,\ldots,n$.
    \end{enumerate}
\end{enumerate}
Recall that $\mathbb{L}^i(s,x,v)=\sum^n_{j=1}d^i_j(s)L^j(x,v)$.

\begin{The}\label{thm:short_time}
If $\bA$ is arbitrary and condition $\mbox{\rm (C3)}$ is satisfied, then there exists $C,\bar{t}>0$ such that each $\mathbb{L}^i(s,x,v)$, $i=1,\ldots,n$, is a time-dependent Lagrangian with $\frac d{dt}\mathbb{L}^i(s,x,v)\leqslant C(1+\mathbb{L}^i(s,x,v))$ on $(0,\bar{t})\times\R^n\times\R^n$. In addition the following statements hold true.
\begin{enumerate}[\rm (1)]
	\item For any $i$, $x,y\in\R^n$ and $t\in(0,\bar{t}]$, $u_i(\xi,t)$ admits a minimizer $\xi_i\in\Gamma^t_{x,y}$. Any such a minimizer $\xi_i$ is of class $C^2$ and satisfies Herglotz equation \eqref{eq:Herglotz} on $[0,t]$.
	\item Let $x,y\in\R^n$, $t\in(0,\bar{t}]$ and $\xi_i$ be a minimizer of $u_i(\cdot,t)$, $i=1,\ldots,n$. Set $p^j_i(s)=L^j_v(\xi_i(s),\dot{\xi}_i(s))$, $i,j=1,\ldots,n$. Then $\{\xi_i\}$ and $\{p^j_i\}$ satisfy Lie equation \eqref{eq:Lie1} on $[0,t]$
	\item For $\phi_1,\ldots,\phi_n\in BUC(\R^n,\R)$ and let $u^1,\ldots,u^n$ be the value functions of the Bolza problem \eqref{eq:Bolza}. Then, each $u^i$ is locally semiconcave on $(0,\bar{t})\times\R^n$ and $(u^1,\ldots,u^n)$ is the unique set of viscosity solutions of \eqref{eq:HJS} on $[0,\bar{t}]\times\R^n$.
\end{enumerate}
\end{The}

\begin{proof}
Observe that for $s\ll1$, we have
\begin{align*}
	\sum^n_{j=1}d^i_j(s)L^j(x,v)=&\,d^i_i(s)L^i(x,v)+\sum^n_{j\not=i}d^i_j(s)L^j(x,v)\\
	\geqslant&\,d^i_i(s)L^i(x,v)-\sum^n_{j\not=i}|d^i_j(s)|\cdot\theta_1(|v|)\\
	\geqslant&\,d^i_i(s)(\theta_0(|v|)-c_0)-\sum^n_{j\not=i}|d^i_j(s)|(C_1+C_1\theta_0(|v|))\\
	=&\,\big\{d^i_i(s)-C_1\sum^n_{j\not=i}|d^i_j(s)|\big\}\theta_0(|v|)-\big\{C_1\sum^n_{j\not=i}|d^i_j(s)|+c_0|d^i_i(s)|\big\}.
\end{align*}
Now, take $t_0>0$ and $c_1,\kappa_1>0$ such that 
\begin{align*}
	d^i_i(s)-C_1\sum^n_{j\not=i}|d^i_j(s)|\geqslant\kappa_1,\quad C_1\sum^n_{j\not=i}|d^i_j(s)|+c_0|d^i_i(s)|\leqslant c_1, \qquad \forall s\in[0,t_0].
\end{align*}
In follows for each $i$
\begin{equation}\label{eq:Tonelli1}
	\sum^n_{j=1}d^i_j(s)L^j(x,v)\geqslant \kappa_1\theta_0(|v|)-c_1,\qquad (x,v)\in\R^n\times\R^n,\ s\in[0,t_0].
\end{equation}

Similarly, in the sense of distribution we have
\begin{align*}
	\sum^n_{j=1}d^i_j(s)D^2_vL^j(x,v)=&\,d^i_i(s)D^2_vL^i(x,v)+\sum^n_{j\not=i}d^i_j(s)D^2_vL^j(x,v)\\
	\geqslant&\,d^i_i(s)D^2_vL^i(x,v)-C_2\sum^n_{j\not=i}|d^i_j(s)|D^2_vL^i(x,v)\\
	=&\,(d^i_i(s)-C_2\sum^n_{j\not=i}|d^i_j(s)|)D^2_vL^i(x,v).
\end{align*}
Similarly, take $t_1>0$ and $\kappa_2>0$ such that 
\begin{align*}
	d^i_i(s)-C_2\sum^n_{j\not=i}|d^i_j(s)|\geqslant\kappa_2,\qquad \forall s\in[0,t_1].
\end{align*}
In follows for each $i$
\begin{equation}\label{eq:Tonelli2}
	D^2_v\bigg(\sum^n_{j=1}d^i_j(s)L^j(x,v)\bigg)\geqslant \kappa_2D^2_vL^i(x,v)>0,\qquad (x,v)\in\R^n\times\R^n,\ s\in[0,t_1].
\end{equation}
Thus, each $\mathbb{L}^i$ is a time-dependent Tonelli Lagrangian by \eqref{eq:Tonelli1} and \eqref{eq:Tonelli2}. The fact $\frac d{dt}\mathbb{L}^i(s,x,v)\leqslant C(1+\mathbb{L}^i(s,x,v))$ on $(0,\bar{t})\times\R^n\times\R^n$ shows there is no Lavrentiev phenomenon and we have full regularity for the minimizers. 

The rest part is directly from the proofs of Theorem \ref{thm:Herglotz}, Theorem \ref{thm:Lie} and Theorem \ref{thm:viscosity} respectively. This completes the proof.
\end{proof}

\begin{Rem}
To extend Theorem \ref{thm:short_time} to a large time, we need use the technique of \emph{broken geodesic} by the conjunction of a sequence of fundamental solutions, which is based on the local analysis of the fundamental solutions such as semiconcavity and convexity estimate. We will touch this part in the future.
\end{Rem}


\appendix

\section{Kamke-M\"uller condition and irreducible matrix}

It is already known that Kamke-M\"uller condition is closely related to the theory of non-negative matrix which plays an essential role in the analysis of evolutionary differential or integral equations of monotone type especially on the long time behavior (\cite{Kamke1932,Muller1927,Smith_book,Walter_book1970}). A good reference on non-negative matrix is \cite{Berman_Plemmons_book}.

\begin{defn}
\hfill
\begin{enumerate}[(1)]
	\item An $n\times n$ matrix $A$ is said to be \emph{non-negative} if all the entries are non-negative real numbers.
	\item An $n\times n$ matrix $A$ is \emph{cogredient} to a matrix $E$ if for some permutation matrix $P$, $PAP^T=E$. $A$ is \emph{reducible} if it is cogredient to
	\begin{align*}
		E=\begin{bmatrix}
			B&0\\
			C&D
		\end{bmatrix},
	\end{align*}
	where $B$ and $D$ are square matrices, or if $n=1$ and $A=0$. 
	\item An $n\times n$ matrix $A$ is said to be \emph{irreducible} if $A$ is not reducible.
	\item An $n\times n$ real matrix $A=(a^i_j)$ is \emph{essentially nonnegative} if $a^i_j\geqslant 0$ for all $i\not=j$.
\end{enumerate}
\end{defn}

\begin{Pro}
Let $A$ be an $n\times n$ real matrix. Then $e^{At}\geqslant0$ for all $t\geqslant0$ if and only if $A$ is essentially nonnegative. In addition, if $A$ is also irreducible then all the entries of $e^{At}$ are positive for all $t>0$.
\end{Pro}

\bibliographystyle{plain}
\bibliography{mybib}

\end{document}